\documentclass[12pt]{amsart}
\usepackage{a4wide} 
\usepackage{latexsym,amsxtra,amscd,ifthen}
\usepackage{amsfonts}
\usepackage{verbatim}
\usepackage{amsmath}
\usepackage{amsthm}
\usepackage{amssymb}

\numberwithin{equation}{section}

\theoremstyle{plain}
\newtheorem{theorem}{Theorem}[section]
\newtheorem{lemma}[theorem]{Lemma}
\newtheorem{proposition}[theorem]{Proposition}

\theoremstyle{definition}

\newtheorem{example}[theorem]{Example}

\newtheorem*{conventions*}{Conventions}
\newtheorem{remark}[theorem]{Remark}
\newtheorem*{remark*}{Remark}

\newtheorem{subsec}[theorem]{}

\def\cl{{\mathcal L}}
\def\cm{{\mathcal M}}
\def\cu{{\mathcal U}}

\def\mr{{\mathbb R}}

\def\bfc{{\boldsymbol c}}
\def\bfr{{\boldsymbol r}}

\def\Atil{\widetilde{A}}

\def\Ltil{\widetilde{L}}

\def\Util{\widetilde{U}}

\def\Ldtil{\overset{\approx}{L}}
\def\Udtil{\overset{\approx}{U}}
\def\ldtil{\overset{\approx}{l}}
\def\udtil{\overset{\approx}{u}}

\def\OqMtwothreemodboxbox{\Oq(M_{2,3}(k))/ \langle \square
   \kern-.1667em \square \rangle}

\def\rank{\operatorname{Rank}}

\begin{document}

\title[LU decomposition of totally nonnegative matrices]{LU decomposition of totally nonnegative matrices}

\author{K.R. Goodearl}

\address{Department of Mathematics, University of California at Santa Barbara, Santa Barbara, CA 93106, USA}

\email{goodearl@math.ucsb.edu} 

\author{T.H. Lenagan}

\address{Maxwell Institute for Mathematical Sciences, School of Mathematics, University of Edinburgh, JCMB, King's Buildings, Mayfield Road, Edinburgh EH9 3JZ, Scotland, UK}

\email{tom@maths.ed.ac.uk}

\thanks{The research of the first named author was supported
by a grant from the National Science Foundation (USA)}

\subjclass[2000]{Primary 15B48}

\keywords{Totally nonnegative matrices, LU decomposition}

\begin{abstract} A uniqueness theorem for an LU decomposition of a
totally nonnegative matrix is obtained.
\end{abstract}

\maketitle

\setcounter{section}{-1}
\section{Introduction} \label{intro}

An $m\times n$ matrix $M$ with entries from $\mr$ is said to be \emph{totally
nonnegative} if each of its minors is nonnegative. Further, such a matrix is
\emph{totally positive} if each of its minors is strictly positive. (Warning:
in some texts, the terms totally positive and strictly totally positive are
used for our terms totally nonnegative and totally positive, respectively.) 

Totally nonnegative matrices arise in many areas of mathematics and there has
been considerable interest lately in the study of these matrices. For
background information and historical references, there is the newly published
book by Pinkus, \cite{pinkus} and also two good survey articles \cite{ando}
and \cite{fz}. 

In this paper, we are interested in the $LU$ decomposition theory of totally
nonnegative matrices. Cryer, \cite[Theorem 1.1]{cryer2}, has proved that any
totally nonnegative matrix $A$ has a decomposition $A=LU$ with $L$ totally
nonnegative lower triangular and $U$ totally nonnegative upper triangular. If,
in addition, $A$ is square and nonsingular then this decomposition is
essentially unique, see, for example, \cite[pages 50-55]{pinkus}, especially
Theorem 2.10 and Proposition 2.11. However, in the singular case such
decompositions need not be unique, as is pointed out in \cite[Page 91]{cryer1}. 

The aim in this paper is to refine the methods of Cryer, \cite{cryer1,
cryer2}, and Gasca and Pe\~na, \cite{gp}, to produce an $LU$ decomposition for
which there is a uniqueness result. 

A short word concerning the genesis of this result may be interesting to
readers. In a series of recent papers, \cite{gll1, gll2, ll}, a very close
connection has emerged between the theory of totally nonnegative matrices and
the theory of the torus invariant prime ideals of the algebra of quantum
matrices. This opens up the possibility of using results and methods from one
of these areas to produce results in the other. The existence of the results
of this paper was suggested by the tensor product decomposition theorem for
torus invariant prime ideals in quantum matrices obtained in an earlier paper
of the present authors, \cite[Theorem 3.5]{gl-ijm}.

\begin{conventions*} If a matrix is denoted by a given capital Roman letter, its entries will be denoted by the corresponding lower case letter, with subscripts. E.g., the entries of a matrix named $L$ will be denoted $l_{ij}$.

When writing sets of row or column indices, we assume that the indices have been listed in strictly ascending order. 

Recall the standard partial order on index sets of the same cardinality, say  $I:=\{i_1,\dots,i_s\}$ and 
$I':=\{i_1',\dots,i_s'\}$, where $i_1<i_2<\cdots<i_s$ and $i_1'<i_2'<\cdots<i_s' $ according to our convention above. Then: $I\leq I'$ 
if and only if $i_k\le i_k'$ for each 
$k=1,\dots,s$.

If $A$ is a matrix and $I$, $J$ are subsets of row indices and column indices for
$A$ then $A(I,J)$ denotes the submatrix of $A$ obtained by using the rows indexed by $I$
and columns indexed by $J$. If $|I|=|J|$, the minor determined by $A(I,J)$, that is, 
${\rm Det}(A(I,J))$, is denoted by 
$[I| J]_A$, or simply by $[I| J]$ if there is no danger of confusion. By convention, $[\varnothing|\varnothing]_A := 1$ for any matrix $A$.
\end{conventions*}


\section{$LU$ decomposition with specified echelon forms}  \label{section-LU}

We begin by giving an $LU$ decomposition for certain rectangular matrices, in which the matrices $L$ (respectively, $U$) have specified lower (respectively, upper) echelon forms. The specification of the matrices for which this decomposition holds, and the decomposition itself, hold over arbitrary fields, and we keep that generality for this section. In Section \ref{section-neville}, we shall prove that all totally nonnegative real matrices satisfy the required hypotheses, and that for such matrices, the resulting factors $L$ and $U$ are also totally nonnegative (see Theorem \ref{theorem-tnn}).


\begin{subsec}{\bf Echelon forms.} 
We say that a matrix $U= (u_{ij})$ is in \emph{upper echelon form} (or \emph{row echelon form}) if the
following hold:
\begin{enumerate}
\item If the $i$th row of $U$ is nonzero and $u_{ij}$ is the leftmost nonzero entry in this row, then $u_{kl}=0$ whenever both $k>i$ and $l\le j$; 
\item If the $i$th row of $U$ is zero then all the rows below it
are zero.
\end{enumerate}
If, in addition to (1) and (2), there are no zero rows then we say that $U$ is in
\emph{strictly upper echelon form}. 

Similar definitions are made for lower triangular matrices. Namely, a matrix $L= (l_{ij})$ is in \emph{lower echelon form} provided the transpose of $L$ is in upper echelon form, that is:
\begin{enumerate}
\item If the $j$th
column of $L$ is nonzero and $l_{ij}$ is the uppermost
nonzero entry in this column, then $l_{kl}=0$ whenever $k\le i$ and $l>j$; 
\item If the $j$th
column of $L$ is zero then all the columns to the right of it are zero.
\end{enumerate} 
If, in addition to (1) and (2),
there are no zero columns then we say that $L$ is in \emph{strictly lower
echelon form}.

In order to obtain the desired uniqueness results, we need to be more precise
concerning the echelon shapes of matrices as above. 
Let $\bfr  := \{r_1,r_2,\dots,r_t\}$ and 
${\bfc}:=\{c_1,c_2,\dots,c_t\}$, where $1\leq r_1<r_2<\dots<r_t\leq m$ and $1\le c_1<\cdots <c_t\le n$. 
\begin{enumerate}
\item We say that an $m\times t$ 
matrix $L= (l_{ij})$ is in the class $\cl_{\bfr }$ provided that for all $j=1,\dots,t$, we have  $l_{r_jj}\neq 0$ and $l_{ij}=0$ for all $i<r_j$. Further, $L\in\cl_{\bfr }^*$ if also
$l_{r_jj}=1$ for all $j$. Note that all the matrices in $\cl_\bfr$ are in strictly lower echelon form.
\item Similarly, we say that a $t\times n$ matrix $U= (u_{ij})$ 
is in the class $\cu_{\bfc}$ provided that for all $i=1,\dots,t$, we have $u_{ic_i}\neq 0$
and $u_{ij}=0$ for all $j<c_i$. All such matrices are in strictly upper echelon form.
\end{enumerate}
\end{subsec}


\begin{subsec}{\bf Some classes of matrices}

Let $\bfr:= \{r_1,\dots,r_t\}$ and $\bfc:= \{c_1,\dots,c_t\}$ be subsets of $\{1,\dots,m\}$ and $\{1,\dots,n\}$, respectively.
An $m\times n$ 
matrix $A$ is said to be in 
the class $\cm_{{\bfr},{\bfc}}$ provided that 
\begin{enumerate}
\item $\rank(A)=t$;
\item For each $s$ with $s\leq t$, the minor 
$[r_1,r_2,\dots,r_s| c_1,c_2,\dots,c_s]_A$ is nonzero;
\item $[I| J]_A=0$ whenever $|I|=|J|= s\le t$ and either $I \ngeq
\{r_1,\dots,r_s\}$ or $J \ngeq \{c_1,\dots,c_s\}$. 
\end{enumerate}
\end{subsec}

\begin{remark} \label{remark-uniqueclass} 
It is easy to check that a matrix belongs to at most one class 
$\cm_{{\bfr},{\bfc}}$. However, 
in  general, a matrix need not belong to any such class --
consider, for example, the matrix 
$A:= 
\left( 
    \begin{array}{cc} 
    0&1\\
    1&1
    \end{array} 
\right).$

Suppose that $L\in \cl_{\bfr}$ where ${\bfr}:= \{r_1,\dots,r_t\}$. 
Note that 
$$[r_1,\dots,r_s|1,\dots,s]_L =l_{r_1 1}\cdots l_{r_s s}\neq 0,$$
for each $s\leq t$. In particular, $[r_1,\dots,r_t|1,\dots,t]_L \neq 0$, so that $\rank(L)=t$.

Suppose that 
$\{i_1, \dots, i_s\}\ngeq \{r_1,\dots,r_s\}$. Then $i_k<r_k$ for some 
$k$. Thus, any submatrix of the form
$L(\{i_1, \dots, i_s\},J)$ is a lower 
triangular matrix with a zero in the $k$th position on the diagonal; 
and so $[i_1, \dots, i_s|J]_L=0$. Since all $s$-element index sets $J \subseteq \{1,\dots,t\}$ satisfy $J \ge \{1,\dots,s\}$, we thus see that
$L\in \cm_{{\bfr},{[1,t]}}$, where ${[1, t]}:=\{1,\dots,t\}$.

Similarly, any $U\in\cu_{\bfc}$ belongs to $\cm_{[1,t],\bfc}$, where $t= |\bfc|$.
\end{remark}

\begin{lemma} \label{lemma-LrUcinMrc}
Suppose that $L$ is an $m\times t$ matrix in the class
$\cl_{\bfr}$ and that $U$ is a $t\times n$ matrix in
the class $\cu_{\bfc}$. 

{\rm (i)} Let $s\le t$ and let $I$ {\rm(}respectively, $J${\rm)} be an $s$-element subset of $\{1,\dots,m\}$ {\rm(}respectively, $\{1,\dots,t\}${\rm)}. Then $[r_1,\dots,r_s|J]_L \ne 0$ if and only if $J= \{1,\dots,s\}$, and $[I|J]_L=0$ if $I\ngeq \{r_1,\dots,r_s\}$.

{\rm (ii)} Let $s\le t$ and let $I$ {\rm(}respectively, $J${\rm)} be an $s$-element subset of $\{1,\dots,t\}$ {\rm(}respectively, $\{1,\dots,n\}${\rm)}. Then $[I|c_1,\dots,c_s]_U \ne 0$ if and only if $I= \{1,\dots,s\}$, and $[I|J]_U=0$ if $J \ngeq \{c_1,\dots,c_s\}$.

{\rm (iii)} $A:=LU$ is an $m\times n$ matrix in the class $\cm_{{\bfr},{\bfc}}$. 
\end{lemma} 

\begin{proof} (i) We already have $[r_1,\dots,r_s|1,\dots,s]_L \ne 0$, and $[I|J]_L=0$ for $I\ngeq \{r_1,\dots,r_s\}$, by Remark \ref{remark-uniqueclass}. If  $J= \{j_1,\dots,j_s\}$ and  $J\ne \{1,\dots,s\}$, then some $j_k>k$, whence $r_{j_k}> r_k$. In this case, $L(\{r_1,\dots,r_s\},J)$ is a lower triangular matrix whose $k,k$ entry is zero, and so $[r_1,\dots,r_s|J]_L =0$.

(ii) This is proved symmetrically.

(iii) First, $\rank(A)\leq t$, as $A$ is the product of an
$m\times t$ matrix and a $t\times n$ matrix. However, by the Cauchy-Binet identity
(Lemma~\ref{lemma-cauchy-binet}), 
\[
[r_1,\dots,r_t|c_1,\dots,c_t]_A\quad
=\quad [r_1,\dots,r_t|1,\dots,t]_L[1,\dots,t|c_1,\dots,c_t]_U \ne 0,
\]
so $\rank(A)=t$. 

For any $s\leq t$, by the Cauchy-Binet identity together with (i), 
\begin{eqnarray*}
[r_1,\dots,r_s|c_1,\dots,c_s]_A
&=&
\sum_K\, [r_1,\dots,r_s|K]_L[K|c_1,\dots,c_s]_U \\
&=& 
[r_1,\dots,r_s|1,\dots, s]_L[1,\dots,s|c_1,\dots,c_s]_U \ne 0.
\end{eqnarray*} 

Now, suppose that we have a row index set $I \ngeq 
\{r_1,\dots,r_s\}$. For any $s$-element subset $K$ of $\{1,\dots,t\}$, we have $[I|K]_L =0$ by (i), and therefore, for any $s$-element subset $J$ of $\{1,\dots,n\}$, Lemma~\ref{lemma-cauchy-binet} implies that $[I|J]_A
=
\sum_K\, [I|K]_L[K|J]_U =0$. Similarly, $[I|J]_A =0$ for any $I$, $J$ with $|I|= |J|= s\le t$ and $J \ngeq \{c_1,\dots,c_s\}$. Therefore $A\in \cm_{{\bfr},{\bfc}}$.
\end{proof}


The following theorem gives an explicit $LU$ decomposition for matrices in the classes $\cm_{\bfr, \bfc}$. Uniqueness of these decompositions will be proved once existence has been established.

\begin{theorem} \label{theorem-explicit} 
Let $A$ be an $m\times n$ matrix which belongs
to the class $\cm_{\bfr, \bfc}$ where ${\bfr} :=\{r_1,\dots,r_t\}$ and ${\bf
c} :=\{c_1,\dots,c_t\}$. 

Set $L:=(l_{ij})$ and $U:=(u_{ij})$ to be the $m\times t$ and $t\times n$ matrices, respectively, 
with entries as follows: $l_{ij} :=0$ for $i<r_j$  and
\[ l_{ij} := [r_1,r_2,\dots,r_{j-1},i|c_1,c_2,\dots,c_j]_A
[r_1,r_2,\dots,r_j|c_1,c_2,\dots,c_j]_A^{-1}
\]
for $i\geq r_j$, while $u_{ij} := 0$ for $j<c_i$ and 
\[
u_{ij} := [r_1,r_2,\dots,r_i|c_1,c_2,\dots,c_{i-1},j]_A
[r_1,r_2,\dots,r_{i-1}|c_1,c_2,\dots,c_{i-1}]_A^{-1}
\]
for $j\geq c_i$.

Then $L$ belongs to the
class $\cl_{\bfr}^*$, while 
$U$ belongs to the
class $\cu_{\bfc}$, and
$A=LU$. 
\end{theorem} 

\begin{proof}
It is obvious from the definitions that $L\in\cl_{\bfr}^*$ and 
$U\in\cu_{\bfc}$; so we need to prove that $A=LU$. 
The proof is by induction on $\min\{m,n\}$ with
the cases where $m=1$ or $n=1$ being trivial.  In this proof, any 
minor $[I|J]$ without a subscript is a minor of $A$; that is, 
$[I|J]=[I|J]_A$. Minors of other matrices are given subscripts. 

Assume that $m,n\ge2$, and suppose first that $a_{11}=0$. Then either $r_1>1$ or $c_1>1$. It follows that
either the first row or first column of $A$ is zero, because $A\in\cm_{\bf
r,\bfc}$. Suppose that the first row of $A$ is zero, in which case 
$r_1>1$. Let 
$\Atil$ be the $(m-1)\times n$ matrix obtained from $A$ by deleting 
the first row.
Then $\Atil\in \cm_{\bfr',\bfc}$ where 
${\bfr}' := \{r_1-1,\dots,r_t-1\}$. By using the inductive hypothesis, there are 
matrices $\Ltil$, $\Util$, with entries as specified above, 
such that $\Atil=\Ltil\Util$. Note that $\Util=U$.

Now, 
\[
A= 
 \left( 
    \begin{array}{ccc} 
    0&\cdots&0\\  \hline
    &&\\
    &\Atil&\\
    &&\\
        \end{array} 
\right)=
\left( 
    \begin{array}{ccc} 
    0&\cdots&0\\  \hline
        &&\\
    &\Ltil&\\
    &&\\
        \end{array} 
\right)
\Util
\]
and it is easy to check that $\left( 
    \begin{array}{ccc} 
    0&\cdots&0\\  \hline
        &&\\
    &\Ltil&\\
    &&\\
        \end{array} 
\right) = L$.

The case where the first column of $A$ is zero is dealt with in a similar
way.

Next, assume that $a_{11}\neq 0$ and note that $r_1=c_1=1$ in this case. 
Then, by elementary row operations using $a_{11}$ as the pivot, we see that
$A=\Ltil\Atil$, where

\[
 \Ltil=
\left( 
    \begin{array}{c|ccc} 
    1&0&\cdots&0\\  \hline
    a_{21}a_{11}^{-1}&&&\\
    \vdots&&I&\\
    a_{m1}a_{11}^{-1}&&&
    \end{array} 
\right) ,
\quad \qquad
\Atil=
\left( 
    \begin{array}{c|ccc} 
    a_{11}&a_{12}&\cdots&a_{1n}\\  \hline 
    0&&&\\
    \vdots&&D&\\
    0&&&
    \end{array} 
\right) ,
\]
and $D=(d_{ij})$ is the
$(m-1)\times (n-1)$ matrix with entries
$$d_{ij} := a_{i+1,j+1}-a_{i+1,1}a_{11}^{-1}a_{1,j+1} = 
[1,i+1|1,j+1][1|1]^{-1}\,.$$
Also, set  $B:= \bigl( [1,i+1|1,j+1] \bigr)$, so that $D= [1|1]^{-1} B$. 

Let $\{i_1,\dots,i_s\}$ and  $\{j_1,\dots,j_s\}$ be subsets of $\{1,\dots,t-1\}$. Then 
\[
[i_1,\dots, i_s|j_1,\dots,j_s]_B = [1 ,i_1+1,\dots, i_s+1 |1 ,j_1+1,\dots,j_s+1][1|1]^{s-1} \,,
\]
by Sylvester's identity (Lemma \ref{lemma-sylvester}). 
It follows that 
\begin{equation}
  \tag{*}
 \begin{aligned}
{}[i_1,\dots, i_s |j_1,\dots,j_s]_D &= [i_1,\dots, i_s |j_1,\dots,j_s]_B [1|1]^{-s}   \\
 &= [1 ,i_1+1,\dots, i_s+1 |1 ,j_1+1,\dots,j_s+1] [1|1]^{-1} \,. 
\end{aligned}
 \end{equation}
From this, it follows that 
$D$ belongs to the class 
$\cm_{\bfr', \bfc'}$ where 
${\bfr'} :=\{r_2-1,\dots,r_t-1\}$ and ${\bfc'} := \{c_2-1,\dots,c_t-1\}$.

By induction, there are $(m-1)\times(t-1)$ and $(t-1)\times(n-1)$ matrices $\Ldtil=(\ldtil_{ij})$ and $\Udtil= (\udtil_{ij})$ such that $D=\Ldtil\Udtil$, with $\ldtil_{ij} =0= l_{i+1,j+1}$ for $i<r_{j+1}-1$ and
\begin{align*}
\ldtil_{ij} &= [r_2{-}1,\dots,r_j{-}1,i | c_2{-}1,\dots,c_{j+1}{-}1]_D [r_2{-}1,\dots,r_{j+1}{-}1 | c_2{-}1,\dots,c_{j+1}{-}1]_D^{-1}  \\
 &= [1,r_2,\dots,r_j,i{+}1 | 1,c_2,\dots,c_{j+1}] [1|1]^{-1} [1,r_2,\dots,r_{j+1} | 1,c_2,\dots,c_{j+1}]^{-1} [1|1]  \\
 &= l_{i+1,j+1}
 \end{align*}
for $i\ge r_{j+1}-1$; while $\udtil_{ij}=0= u_{i+1,j+1}$ for $j<c_{i+1}-1$ and
\begin{align*}
\udtil_{ij} &= [r_2{-}1,\dots,r_{i+1}{-}1 | c_2{-}1,\dots,c_i{-}1,j]_D [r_2{-}1,\dots,r_i{-}1 | c_2{-}1,\dots,c_i{-}1]_D^{-1}  \\
 &= [1,r_2,\dots,r_{i+1} | 1,c_2,\dots,c_i,j{+}1] [1|1]^{-1} [1,r_2,\dots,r_i | 1,c_2,\dots,c_i]^{-1} [1|1]  \\
 &= u_{i+1,j+1}
 \end{align*}
for $j\ge c_{i+1}-1$, by using $(*)$ above.

Now, observe that 
$$A= \Ltil\Atil=
\left( 
    \begin{array}{c|ccc} 
    1&0&\cdots&0\\  \hline
    a_{21}a_{11}^{-1}&&&\\
    \vdots&&I&\\
    a_{m1}a_{11}^{-1}&&&
    \end{array} 
\right)
\left( 
    \begin{array}{c|ccc} 
    1&0&\cdots&0\\  \hline
    0&&&\\
    \vdots&&\Ldtil&\\
    0&&&
    \end{array} 
\right)
\left( 
    \begin{array}{c|ccc} 
    a_{11}&a_{12}&\cdots&a_{1n}\\  \hline 
    0&&&\\
    \vdots&&\Udtil&\\
    0&&&
    \end{array} 
\right)  .$$
From our calculations of the entries of $\Ldtil$ and $\Udtil$ above, we see that
\[
L=
\left( 
    \begin{array}{c|ccc} 
    1&0&\cdots&0\\  \hline
    a_{21}a_{11}^{-1}&&&\\
    \vdots&&\Ldtil&\\
    a_{m1}a_{11}^{-1}&&&
    \end{array} 
\right)
\qquad\quad{\rm and}\qquad\quad U=\left( 
    \begin{array}{c|ccc} 
    a_{11}&a_{12}&\cdots&a_{1n}\\  \hline 
    0&&&\\
    \vdots&&\Udtil&\\
    0&&&
    \end{array} 
\right),
\]
and therefore $A=LU$. This completes the inductive step.
\end{proof}


\begin{theorem} \label{theorem-uniquelu}
Suppose that $A$ is in the class $\cm_{{\bfr},{\bfc}}$ 
where ${\bfr} := \{r_1,r_2,\dots,r_t\}$ and 
${\bfc}:=\{c_1,c_2,\dots,c_t\}$.  There are unique matrices $L\in \cl_{\bfr}^*$ and $U\in \cu_{\bfc}$ such that $A=LU$, namely those given in Theorem {\rm\ref{theorem-explicit}}.
\end{theorem} 

\begin{proof} 
We show that if $A= (a_{ij})$ is in the class $\cm_{{\bfr},{\bfc}}$  and $A=LU$ with 
$L= (l_{ij}) \in \cl_{\bfr}^*$ and $U= (u_{ij}) \in \cu_{\bfc}$, then the entries 
of $L$ and $U$ can be uniquely
specified from this information. The result then follows.

Note that the equations
\[
a_{r_1j}=(LU)_{r_1j}=\sum_{k=1}^t\,l_{r_1k}u_{kj}=l_{r_11}u_{1j}
=u_{1j}
\]
specify the first row of $U$. 
Similarly, 
\[
a_{ic_1} = (LU)_{ic_1}= 
        \sum_{k=1}^t\,l_{ik}u_{kc_1}=l_{i1}u_{1c_1}
        \]
for all $i$, which specifies the first column of $L$, as $u_{1c_1}\neq 0$. 
        
Assume as an  inductive hypothesis 
that the first $s$ rows of $U$ and the first $s$ columns 
of $L$ have been specified. 
Then 
\begin{align*}
a_{r_{s+1}j} &=(LU)_{r_{s+1}j}
=\sum_{k=1}^t\,l_{r_{s+1}k}u_{kj}
=\sum_{k=1}^s\,l_{r_{s+1}k}u_{kj}+l_{r_{s+1}s+1}u_{s+1 , j}  \\
 &= \sum_{k=1}^s\,l_{r_{s+1}k}u_{kj}+u_{s+1 , j}
\end{align*}
for all $j$, which specifies the $(s+1)$-st row of $U$ because the terms in the last
summation are already known by induction. 

Finally, 
\[
a_{ic_{s+1}}=(LU)_{ic_{s+1}}
=\sum_{k=1}^t\,l_{ik}u_{k,c_{s+1}}
=\sum_{k=1}^s\,l_{ik}u_{k,c_{s+1}}+l_{i, s+1}u_{s+1 , c_{s+1}}
\]
for all $i$, which specifies the $(s+1)$-st column of $L$ because the terms in the 
summation  on the 
right are already known by induction and $u_{s+1 , c_{s+1}}\neq 0$. 

This finishes the inductive step and so the result is proved.  
\end{proof}


\section{Modified Neville elimination} \label{section-neville}

We now restrict attention to real matrices and focus on total nonnegativity. Our aim is to show that any $m\times n$ totally nonnegative matrix $A$ lies in one of the classes $\cm_{{\bfr},{\bfc}}$, and that the matrices $L$ and $U$ in the decomposition $A=LU$ of Theorem \ref{theorem-explicit} are totally nonnegative. While it is possible to prove directly that these matrices are totally nonnegative, it is technically much less complicated to obtain that result via a modification of the Neville elimination process of Gasca and Pe\~na,
\cite{gp}.  

The following is an elementary, but crucial, fact about totally nonnegative 
matrices.

\begin{lemma} \label{lemma-cauchon} 
Suppose that $A= (a_{ij})$ is a totally nonnegative $m\times n$ matrix, and that $i<k\le m$ and $j<l\le n$. 
If $a_{ij}=0$ then either $a_{il}=0$ or $a_{kj}=0$. As a consequence, 
if some entry in the $j$th column, but below the $ij$ entry, is nonzero then 
all elements in the $i$th row, but to the right of the $ij$ entry, are also zero. 
\end{lemma}

\begin{proof} 
Note that $0\leq [ik|jl] = a_{ij}a_{kl}-a_{il}a_{kj}= -a_{il}a_{kj}$; so that
$a_{il}a_{kj}\leq 0$. As $a_{il},a_{kj}\geq 0$ this gives the desired
conclusion. 
\end{proof} 

First, we give an informal description of the elimination process that we will
use. We start with a totally nonnegative matrix. The aim is to use a version
of the Neville elimination procedure to produce a final matrix in echelon form
with no zero rows. If a zero row appears at any stage in the process then we
delete it (rather than moving it to the bottom as in ordinary Neville
elimination). Otherwise, we proceed as with Neville elimination: if we are
clearing the lower entries in a given column and want to perform a row
operation to replace the last nonzero entry in a column by zero, then we
perform a row operation by subtracting a suitable multiple of the row
immediately above this last position. Note that the entry immediately above
this last position will be nonzero: this is guaranteed by the above lemma. In
the end we produce an upper triangular matrix $U$ in echelon form which
contains no zero rows. Keeping track of the operations performed produces a
lower triangular matrix $L$ such that $A=LU$. We also show that each of $L$
and $U$ is totally nonnegative. \\


\begin{subsec}{\bf Invariants of the  elimination algorithm} 

The modified Neville algorithm starts with $L:= I$ and $U:=A$, 
a totally nonnegative matrix, and 
uses two moves: (i) either delete a row of zeros of $U$ and the 
corresponding column in $L$, or (ii) 
perform a Neville elimination move. 

The first aim is to show that at all times during the modified Neville
algorithm we retain the totally nonnegative condition for $L$ and $U$ 
and the fact that
$A=LU$. There are two moves to consider. The first deletes a row of $U$ and
the corresponding column of $L$. Note that if we delete a row or column from a
totally nonnegative matrix then the new matrix is also totally nonnegative. 
\end{subsec}

\begin{lemma} \label{lemma-delete1} 
Let $B$ be an $m\times p$ matrix and let $C$ be a $p\times n$ matrix. Suppose that 
row $i$ of $C$ is zero. Set $B'$ to be the 
$m\times (p-1)$ matrix obtained by deleting the $i$th column of $B$ and set 
$C'$ to be the $(p-1)\times n$ matrix obtained by deleting the $i$th row of 
$C$. Then $B'C'=BC$.
\end{lemma}

\begin{proof} Obvious. 
\end{proof} 

\begin{lemma} \label{lemma-delete2} 
{\rm(i)} Suppose that ${\bfr} := \{r_1,r_2,\dots,r_t\}$ and that $L\in\cl_{{\bf
r}}.$ Let $L'$ be the matrix obtained by deleting column $i$ from $L$. Then
$L'\in\cl_{{\bfr'}}$ where ${\bfr'} := \{r_1,r_2,\dots, r_{i-1},
r_{i+1},\dots, r_t\}$.

{\rm(ii)} Suppose that ${\bfc} := \{c_1,c_2,\dots,c_t\}$ and that
$U\in\cu_{{\bfc}}.$ Let $U'$ be the matrix obtained by deleting row $i$ from
$U$. Then $U'\in\cu_{{\bfc'}}$ where ${\bfc'} := \{c_1,c_2,\dots, c_{i-1},
c_{i+1},\dots, c_t\}$.

\end{lemma} 

\begin{proof} Obvious. 
\end{proof}


The next results consider the effect of performing 
a Neville elimination move; that is, 
the row operation of
subtracting a suitable multiple of row $s$ from row $s+1$ on the minors of a
matrix of the form 
\[
\left( 
    \begin{array}{ccc|cccc} 
    \star &\cdots&\star&\star&\cdots&\cdots&\star\\
    \vdots&&\vdots&\vdots&&&\vdots\\
    \star &\cdots&\star&\star&\cdots&\cdots&\star\\
    \hline 
    0&\cdots&0&a_{st}&a_{s ,t+1}&\cdots&a_{sn}\\
    0&\cdots&0&a_{s+1 ,t}&a_{s+1 ,t+1}&\cdots&a_{s+1 ,n}\\
    0&\cdots&0&0&\star&\cdots&\star\\
    \vdots&&\vdots&\vdots&&&\vdots\\
    0&\cdots&0&0&\star&\cdots&\star\\
    \end{array} 
\right) 
\] 
when $a_{st}$ and $a_{s+1 ,t}$ are nonzero, in order to clear the entry 
in position $(s+1,t)$. Note that the resulting matrix has the form 
\[
\left( 
    \begin{array}{ccc|cccc} 
    \star &\cdots&\star&\star&\cdots&\cdots&\star\\
    \vdots&&\vdots&\vdots&&&\vdots\\
    \star &\cdots&\star&\star&\cdots&\cdots&\star\\
    \hline 
    0&\cdots&0&a_{st}&a_{s ,t+1}&\cdots&a_{sn}\\
    0&\cdots&0&0&b_{s+1 ,t+1}&\cdots&b_{s+1 ,n}\\
    0&\cdots&0&0&\star&\cdots&\star\\
    \vdots&&\vdots&\vdots&&&\vdots\\
    0&\cdots&0&0&\star&\cdots&\star\\
    \end{array} 
\right) 
\]

\begin{lemma}\label{lemma-nevillestep} 
Suppose that $A=(a_{ij})$ with $a_{st}\neq 0$ and $a_{ij}=0$ whenever $i\geq
s$ and $j<t$. Suppose that $a_{s+1, t}\neq 0$ while 
$a_{s+w , t}=0$ for all $w>1$. Set $B=(b_{ij})$ where
$b_{ij}=a_{ij}$ for $i\neq s+1$ while $b_{s+1  ,j} = a_{s+1, j} - a_{s+1 ,
t}a_{st}^{-1}a_{sj}$ for all $j$. In particular, $b_{s+1,j}=0$ for $j\le t$ while $b_{s+1,j} =[s , s+1|tj]_A a_{st}^{-1}$ for $j>t$. 

Then 
$$[I|J]_B= \begin{cases}  [I |J]_A &{\rm when\;} s\in I {\rm \;or\;} s+1\not\in I   \\ 
      [I |J]_A - a_{s+1 ,t}a_{st}^{-1}[I\backslash\{s+1\}\sqcup\{s\} | J]_A
      &{\rm when\;} s\not\in I {\rm \;and\;} s+1\in I \,.  \end{cases}$$
\end{lemma}
 
\begin{proof} 
Obvious from the definition of $B$. 
\end{proof} 

We refer to the change from $A$ to $B$ described in this lemma as a {\em
Neville elimination move}. The next result shows that the totally nonnegative
condition is preserved under a Neville elimination move. This result may be
well-known, but we have been unable to find a clear statement in the
literature.

\begin{proposition} \label{proposition-tnn} 
In the above setting, if $A$ is totally nonnegative then so is $B$.
\end{proposition}

\begin{proof} 
It follows from the fact that $A$ is totally nonnegative 
and the definition of $B$ that each
$b_{ij}\geq 0$. Also, for any size minor, $[I|J]_B=[I|J]_A\geq 0$ whenever
$s\in I$ or $s+1\notin I$.

Suppose that $l\geq 2$ and that 
all minors of $B$ of size less than $l\times l$ are $\geq 0$. 
Let $[I|J]_B$ be an $l\times l$ minor. By the above remarks, we may assume that $s\not\in I$ and that $s+1\in I$.  
Consider the following cases: 
\begin{enumerate}
\item $s+1$ is the least entry in
$I$; 
\item $s+1\in I$, and there
exists $i\in I$ with $i <s$.
\end{enumerate}

In case (1), consider first the case where there is a $j\in J$ with $j\leq t$. 
Then the $j$th column of $B(I,J)$ is zero; so $[I|J]_B=0$. 
Otherwise, note that 
$$a_{st}[I|J]_B=[I\sqcup \{s\}|J\sqcup \{t\}]_B
=[I\sqcup \{s\}|J\sqcup \{t\}]_A\geq 0.$$
As $a_{st}>0$ it follows that 
$[I|J]_B\geq 0$.

Next, consider case (2). If $[I\backslash\{s+1\}\sqcup\{s\} | J]_A=0$ then 
$[I|J]_B=[I|J]_A\geq 0$, by the previous lemma; so we may assume that 
$$[I\backslash\{s+1\}\sqcup\{s\} | J]_B =
[I\backslash\{s+1\}\sqcup\{s\} | J]_A\neq 0.$$
Suppose that $[I\backslash\{i,s+1\}\sqcup\{s\} | Y]_B=0$ for all subsets $Y$
of $J$ with $|Y|=l-1$. 
Then $[I\backslash\{s+1\}\sqcup\{s\} | J]_B=0$, by Lemma~\ref{lemma-tnn}.
Thus, we may assume that there exists a subset $Y$ of $J$ with $|Y|=l-1$ and
$[I\backslash\{i,s+1\}\sqcup\{s\} | Y]_B> 0$. Suppose that
$J=Y\sqcup\{k\}$. Choose $j\in Y$. 

Apply the Laplace relation of
Lemma~\ref{lemma-laplace}(a) with $J_1=\{j\}$ 
and $J_2=\{j,k\}$ while $I=\{i,s,s+1\}$ to obtain 
\[
[i|j]_B[s ,s+1|jk]_B-[s|j]_B[i ,s+1|jk]_B +[s+1|j]_B[is|jk]_B=0.
\]
It follows that 
\[
[s|j]_B[i ,s+1|jk]_B=
[i|j]_B[s ,s+1|jk]_B+[s+1|j]_B[is|jk]_B \,.
\]

By using Muir's law of extensible minors (Lemma \ref{lemma-muir}), we may introduce the $l-2$ row 
indices from $I\backslash\{i,s+1\}$ and the $l-2$ column indices from 
$Y\backslash\{j\}$ to obtain
\begin{multline*} 
[I\backslash\{i,s+1\}\sqcup\{s\} | Y]_B[I|J]_B=  \\
 [I\backslash\{s+1\}|Y]_B[I\backslash\{i\}\sqcup\{s\} | J]_B
+
[I\backslash\{i\}|Y]_B[I\backslash\{s+1\}\sqcup\{s\} | J]_B \,.
\end{multline*}
Now, $[I\backslash\{i,s+1\}\sqcup\{s\} | Y]_B>0$, by assumption, 
and each of the four 
minors on the right side 
of this equation is $\geq 0$ (the two of size $l-1$ 
by the inductive hypothesis and the two of size $l$ because $s$ is in the 
row set of the minor). 
It follows that $[I|J]_B\geq 0$, as required
\end{proof}

\begin{remark} \label{remark-elementary} 
Let $E(s+1,s)$ be the matrix with $1$ in the $(s+1,s)$ position and 
zero elsewhere.
Note that, with the above notation, 
\[
B=(I-a_{s+1 ,t}a_{st}^{-1}E(s+1,s))A \qquad\text{and}\qquad
A=(I+a_{s+1 ,t}a_{st}^{-1}E(s+1,s))B.
\] 
Note also that $I+a_{s+1 ,t}a_{st}^{-1}E(s+1,s)$
is totally nonnegative. 
\end{remark}


\begin{subsec}{\bf The Modified Neville Algorithm} 

Let $A$ be an $m\times n$ totally nonnegative matrix of rank $t$. 
The following algorithm outputs an $LU$ decomposition of $A$, which, as we shall see, coincides with the one given in Theorem \ref{theorem-explicit}. \\

\noindent\underline{Input}\ \  Set $L:=I$, the identity $m\times m$ matrix,  
and $U:= A$. Note that $A=LU$. \\

\noindent\underline{Algorithm} \\

\noindent\underline{Step 1}\ \  If $U$ is in strictly upper echelon form then stop
and output $L$ and $U$. Otherwise, if there is a row of $U$ consisting
entirely of zeros, go to Step 2 and if not, then go to Step 3. \\

\noindent\underline{Step 2}\ \  Suppose that $L$ is of size $m\times w$ and $U$
of size $w\times n$, and that some row of $U$ is zero.
Choose $i$ as large as possible so that the $i$th row of $U$ is zero. Delete
row $i$ from $U$ and column $i$ from $L$ to obtain new matrices 
$L$ of size $m\times(w-1)$ and $U$ of size $(w-1)\times n$. Note that 
we still have $A=LU$, by Lemma~\ref{lemma-delete1}, and that 
$L$ and $U$ are still totally nonnegative. Go to Step 1. \\

\noindent\underline{Step 3}\ \  Suppose that all rows of $U$ are nonzero, but $U$ is not in upper echelon form. By Lemma \ref{lemma-cauchon}, the leftmost nonzero column of $U$ must have a nonzero entry in its uppermost position. Set $U= (u_{ij})$.

If the first column of $U$ has two or more nonzero
entries then set $t=1$. Otherwise, set $t>1$ so that the submatrix of $U$
consisting of the first $t-1$ columns is in upper echelon form, but that
consisting of the first $t$ columns is not. Then, in view of Lemma \ref{lemma-cauchon}, there is a largest integer $s$ such that $u_{st},u_{s+1 ,t}\neq 0$; moreover, $u_{ij}=0$ for $i\ge s$ and $j<t$. Perform a Neville elimination move on $U$ as in
Lemma~\ref{lemma-nevillestep}; that is, replace $U$ by
$(I-u_{s+1 ,t}u_{st}^{-1}E(s+1,s))U$; so that in the new $U$ we have 
$u_{s+1 ,t}=0$. At the same time, replace $L$
by $L(I+u_{s+1 ,t}u_{st}^{-1}E(s+1,s))$. 
Note that we still have $A=LU$, and that $U$ is totally nonnegative by
Proposition~\ref{proposition-tnn}, while $L$ is the product of two totally
nonnegative matrices and so is still totally nonnegative. 

Go to Step 1.
\end{subsec}

\begin{theorem} \label{theorem-output} 
The above algorithm outputs 
an $m\times t$
totally nonnegative matrix $L\in\cl_{{\bfr}}^*$\,, 
for some $\bfr= \{r_1,r_2,\dots,r_t\}$,
and a $t\times n$ totally
nonnegative matrix $U\in\cu_{{\bfc}}$\,, 
for some  ${\bfc}=\{c_1,c_2,\dots,c_t\}$, 
such that $A=LU$.
\end{theorem} 

\begin{proof} The algorithm outputs totally nonnegative 
matrices $L$ and $U$, in
strictly lower and upper echelon forms, respectively, 
such that $A=LU$. Also, note that the leading entry in each column of $L$ 
is $1$. 
Suppose that $L\in\cl_{{\bfr}}^*$ and $U\in\cu_{{\bfc}}$ with 
${\bfr}=\{r_1,r_2,\dots,r_w\}$ and 
${\bfc}=\{c_1,c_2,\dots,c_w\}$,
As $L$ is an $m\times w$ matrix and $U$
is a $w\times n$ matrix, we have $t= \rank(A) \le w$. Moreover,
$$[r_1,r_2,\dots,r_w| c_1,c_2,\dots,c_w]_A 
= [r_1,r_2,\dots,r_w| 1,\dots,w]_L
[1,\dots,w| c_1,c_2,\dots,c_w]_U\neq 0,$$
by using the Cauchy-Binet identity; 
so $t\ge w$. Hence, $w=t$, as required.
\end{proof} 

The above theorem, combined with the results of Section \ref{section-LU}, yields the main result of the paper:

\begin{theorem} \label{theorem-tnn} 
Let $A$ be an $m\times n$ totally nonnegative matrix. 
Then there is a unique pair ${\bfr}$, ${\bfc}$ such that 
$A\in\cm_{{\bfr},{\bfc}}$. Further, there is then a unique pair 
$L\in\cl_{{\bfr}}^*$, $U\in\cu_{{\bfc}}$ such that $A=LU$. The matrices 
$L$ and $U$ are totally nonnegative.  They are given explicitly in Theorem {\rm\ref{theorem-explicit}}.
\end{theorem} 

\begin{proof} By Theorem \ref{theorem-output}, there exist $\bfr$, $\bfc$ and totally nonnegative matrices $L\in\cl_{{\bfr}}^*$, $U\in\cu_{{\bfc}}$ such that $A=LU$, and $A\in \cm_{\bfr,\bfc}$ by Lemma \ref{lemma-LrUcinMrc}. As noted in Remark \ref{remark-uniqueclass}, $\bfr$ and $\bfc$ are uniquely determined by $A$. The uniqueness of $L$ and $U$ then follows from Theorem \ref{theorem-uniquelu}.
\end{proof} 

Theorem \ref{theorem-explicit} and the total nonnegativity of the factors $L$ and $U$ are known for the case where $A$ is a totally nonnegative 
nonsingular square matrix; see, for example, 
\cite[Theorem 2.10 and Proposition 2.11]{pinkus}.
However, we have not been able to locate a 
prior source for the result just proved.


\section{Examples} 

\begin{example}
We first illustrate the modified Neville algorithm at work on the example considered by Cryer, 
\cite[Page 91]{cryer1}. The matrix in question is 
\[
A:= 
\left( 
    \begin{array}{ccc} 
    0&0&0\\
    1&0&1\\
    1&0&1
    \end{array} 
\right)
\]
Cryer exhibits two distinct $LU$ factorisations of $A$ into 
totally nonnegative factors: 
\[
A= 
\left( 
    \begin{array}{ccc} 
    0&0&0\\
    1&0&1\\
    1&0&1
    \end{array} 
\right)
=
\left( 
    \begin{array}{ccc} 
    0&0&0\\
    1&0&0\\
    1&0&0
    \end{array} 
\right)
\left( 
    \begin{array}{ccc} 
    1&0&1\\
    0&0&0\\
    0&0&0
    \end{array} 
\right)
=
\left( 
    \begin{array}{ccc} 
    0&0&0\\
    1&1&0\\
    1&1&0
    \end{array} 
\right)
\left( 
    \begin{array}{ccc} 
    1&0&0\\
    0&0&1\\
    0&0&1
    \end{array} 
\right)
\]

It is easy to check that $A$ is totally nonnegative of rank one, and that 
$A$ belongs to the class $\cm_{{\{2\}},{\{1\}}}$. 
We start the algorithm 
with the pair $\{I,A\}$: 

\begin{eqnarray*} A=IA&=&
\left( 
    \begin{array}{ccc} 
    1&0&0\\
    0&1&0\\
    0&0&1
    \end{array} 
\right)
\left( 
    \begin{array}{ccc} 
    0&0&0\\
    1&0&1\\
    1&0&1
    \end{array} 
\right)
=
\left( 
    \begin{array}{cc} 
    0&0\\
    1&0\\
    0&1
    \end{array} 
\right)
\left( 
    \begin{array}{ccc} 
    1&0&1\\
    1&0&1
    \end{array} 
\right)
\\[1ex]
&=&
\left( 
    \begin{array}{cc} 
    0&0\\
    1&0\\
    1&1
    \end{array} 
\right)
\left( 
    \begin{array}{ccc} 
    1&0&1\\
    0&0&0
    \end{array} 
\right)
=
\left( 
    \begin{array}{cc} 
    0\\
    1\\
    1
    \end{array} 
\right)
\left( 
    \begin{array}{ccc} 
    1&0&1\\
        \end{array} 
\right)
\end{eqnarray*} 
and one can easily check that 
\[
A=\left( 
    \begin{array}{cc} 
    0\\
    1\\
    1
    \end{array} 
\right)
\left( 
    \begin{array}{ccc} 
    1&0&1\\
        \end{array} 
\right)
\]
is (essentially) the unique decomposition of $A$ as a product 
of a $3\times 1$ matrix and a $1\times 3$ matrix. 
\end{example}


\begin{example}
A more complicated example. The algorithm reveals the class of $A$. 

\begin{eqnarray*} 
A = IA &=& \left( 
    \begin{array}{cccc} 
    1&0&0&0\\
    0&1&0&0\\
    0&0&1&0\\
    0&0&0&1
    \end{array} 
\right)
\left( 
    \begin{array}{cccc} 
    0&1&2&1\\
    0&2&4&2\\
    0&1&2&3\\
    0&3&6&11
    \end{array} 
\right)
=
\left( 
    \begin{array}{cccc} 
    1&0&0&0\\
    0&1&0&0\\
    0&0&1&0\\
    0&0&3&1
    \end{array} 
\right)
\left( 
    \begin{array}{cccc} 
    0&1&2&1\\
    0&2&4&2\\
    0&1&2&3\\
    0&0&0&2
    \end{array} 
\right)\\
&=&
\left( 
    \begin{array}{cccc} 
    1&0&0&0\\
    0&1&0&0\\
    0&1/2&1&0\\
    0&3/2&3&1
    \end{array} 
\right)
\left( 
    \begin{array}{cccc} 
    0&1&2&1\\
    0&2&4&2\\
    0&0&0&2\\
    0&0&0&2
    \end{array} 
\right)
=
\left( 
    \begin{array}{cccc} 
    1&0&0&0\\
    2&1&0&0\\
    1&1/2&1&0\\
    3&3/2&3&1
    \end{array} 
\right)
\left( 
    \begin{array}{cccc} 
    0&1&2&1\\
    0&0&0&0\\
    0&0&0&2\\
    0&0&0&2
    \end{array} 
\right)\\
&=&
\left( 
    \begin{array}{ccc} 
    1&0&0\\
    2&0&0\\
    1&1&0\\
    3&3&1
    \end{array} 
\right)
\left( 
    \begin{array}{cccc} 
    0&1&2&1\\
    0&0&0&2\\
    0&0&0&2
    \end{array} 
\right)
=
\left( 
    \begin{array}{ccc} 
    1&0&0\\
    2&0&0\\
    1&1&0\\
    3&4&1
    \end{array} 
\right)
\left( 
    \begin{array}{cccc} 
    0&1&2&1\\
    0&0&0&2\\
    0&0&0&0
    \end{array} 
\right)\\
&=&
\left( 
    \begin{array}{cc} 
    1&0\\
    2&0\\
    1&1\\
    3&4
    \end{array} 
\right)
\left( 
    \begin{array}{cccc} 
    0&1&2&1\\
    0&0&0&2\\
    \end{array} 
\right) .
\end{eqnarray*}
It follows that $A$ is in $\cm_{{\{1,3\}},{\{2,4\}}}$. 
\end{example}


\section{Appendix: Matrix identities, etc.} \label{section-identities}


For any index sets $I$ and $J$, set $\ell(I,J):=\left|\{(i,j) \in I\times J \mid i>j\}\right|$.


 \begin{lemma} \label{lemma-laplace} 
{\rm (Laplace relations; see, for example, \cite[p14]{mm}, \cite[eqn.~(3.3.4), p26]{veda})} 
Let $A$ be an $m\times n$ matrix, $I\subseteq \{1,\dots,m\}$, and $J\subseteq
\{1,\dots,n\}$. \\

{\rm (a)} If $J_1,J_2\subseteq \{1,\dots,n\}$ with $|J_1|+|J_2| =|I|$,
then
$$\sum_{\substack{ I_1\sqcup I_2=I\\ |I_\nu|=|J_\nu|}}
(-1)^{\ell(I_1;I_2)} [I_1| J_1]_A [I_2| J_2]_A =\begin{cases}
(-1)^{\ell(J_1;J_2)}[I| J_1\sqcup J_2]_A &\quad (J_1\cap J_2=
\varnothing)\\ 0 &\quad (J_1\cap J_2\ne \varnothing). \end{cases}$$

{\rm (b)} If $I_1,I_2\subseteq \{1,\dots,n\}$ with  $|I_1|+|I_2| =|J|$,
then
$$\sum_{\substack{J_1\sqcup J_2=J\\ |J_\nu|=|I_\nu|}}
(-1)^{\ell(J_1;J_2)} [I_1| J_1]_A [I_2| J_2]_A =\begin{cases} 
(-1)^{\ell(I_1;I_2)}[I_1\sqcup I_2| J]_A &\quad (I_1\cap I_2=
\varnothing)\\ 0 &\quad (I_1\cap I_2\ne \varnothing). \end{cases}$$
 \end{lemma} 


 \begin{lemma}\label{lemma-tnn} 
Let $A$ be an $m\times n$ matrix, $I\subseteq \{1,\dots,m\}$, and $J\subseteq
\{1,\dots,n\}$, with
$|I|=|J|$.\\

{\rm (a)} Fix $J_1\subseteq J$. If $[I_1| J_1]_A =0$ for all
$I_1\subseteq I$ with $|I_1|= |J_1|$, then $[I| J]_A =0$.

{\rm (b)} Fix $I_1\subseteq I$. If $[I_1| J_1]_A =0$ for all
$J_1\subseteq J$ with $|J_1|= |I_1|$, then $[I| J]_A =0$.
 \end{lemma}

\begin{proof} 
By symmetry, we need only prove (a). Set $J_2=J\backslash J_1$. There is a 
Laplace relation of the form 
\[
[I| J]_A =\sum_{I_1\sqcup I_2=I}\pm[I_1| J_1]_A [I_2| J_2]_A \,.
\]
As all $[I_1| J_1]_A =0$, by assumption, it follows that $[I| J]_A =0$.
\end{proof}



\begin{lemma} {\rm (Cauchy-Binet Identity; see, for example, \cite[eqn.~(6), p86]{aitken}, \cite[p14]{mm})}
 \label{lemma-cauchy-binet} 
Let $A$ be an $m\times t$ matrix and $B$ a $t\times n$ matrix, and 
let $I\subseteq \{1,\dots,m\}$ and $J\subseteq \{1,\dots,n\}$ be 
$k$-element sets with $k\leq t$. Then
\[
[I|J]_{AB}=\sum_K\, [I|K]_A[K|J]_B
\]
where $K$ ranges over all $k$-element subsets of $\{1,\dots,t\}$. 
\end{lemma}


\begin{lemma} {\rm (Muir's Law of Extensible Minors; see, for example, 
\cite[p179, \S 187]{muir}, \cite[p205]{brsc2})} \label{lemma-muir} 
Let $F$ be a field and suppose that 
$$\sum_{s=1}^d c_s[I_s|J_s][K_s|L_s]=0$$
is a homogeneous determinantal 
 identity for matrices over $F$. Suppose that $P$ is a set of row indices disjoint from each of the 
 sets $I_s$ and $Q$ is a set of column indices disjoint from each of the 
 sets $J_s$, with $|P|=|Q|$. Then 
\[ 
\sum_{s=1}^d c_s[I_s \sqcup P|J_s \sqcup Q]
[K_s \sqcup P|L_s \sqcup Q]=0 
\] 
is also a determinantal identity for matrices over $F$. 
\end{lemma} 


\begin{lemma}  {\rm (Sylvester's Identity; see, for example, 
\cite[p32]{gant}, \cite[eqn.~(8), p772]{brsc1})} \label{lemma-sylvester}
Let $A=(a_{ij})$ be an $n\times n$ matrix and let $m<n$. Set $B=(b_{ij})$ to be the $(n-m)\times(n-m)$ matrix
where $b_{ij} := [1,\dots, m ,m+i |1,\dots, m , m+j]$. Then, 
\[
\det(B)=[1,\dots, n |1,\dots, n]_A [1,\dots, m |1,\dots, m]_A^{n-m-1} \,.
\]
\end{lemma}







\begin{thebibliography}{10}

\bibitem{aitken} A C Aitken, Determinants and Matrices, 9th ed., Interscience, New York, 1956.

\bibitem{ando} T Ando, \emph{Totally positive matrices}, 
Linear Algebra and its Applications \textbf{90} (1987), 165-219.

\bibitem{brsc1} R A Brualdi and H Schneider, \emph{Determinantal identities: Gauss, Schur, Cauchy, Sylvester, Kronecker, Jacobi, Binet, Laplace, Muir, and Cayley}, Linear Algebra and its Applications \textbf{52/53} (1983) 769-781.

\bibitem{brsc2} \bysame, \emph{Determinantal identities revisited}, Linear Algebra and its Applications \textbf{59} (1984) 203-307.

\bibitem{cryer1} C W Cryer, \emph{The $LU$-Factorization of Totally 
Positive Matrices}, Linear Algebra and its Applications \textbf{7} (1973), 83-92. 

\bibitem{cryer2} \bysame, \emph{Some Properties of Totally 
Positive Matrices}, Linear Algebra and its Applications \textbf{15} (1976), 1-25. 

\bibitem{fz} S Fomin and A Zelevinsky, \emph{Total Positivity: tests and
parameterizations}, Math Intelligencer \textbf{22} (2000), no. 1, 23-33. 


\bibitem{gant} F R Gantmacher, The Theory of Matrices, Vol 1, 
Chelsea Publishing Co., New York, 1959.

\bibitem{gp} M Gasca and J M Pe\~na, \emph{Total positivity and Neville 
Elimination}, Linear Algebra and its Applications \textbf{165} (1992), 25-44. 

\bibitem{gll1}K R Goodearl, S Launois and T H Lenagan, \emph{
Torus-invariant prime ideals in quantum matrices, totally
nonnegative cells and symplectic leaves}, Mathematische Zeitschrift, 
doi:10.1007/s00209-010-0714-5. 

\bibitem{gll2} \bysame, \emph{Totally
nonnegative cells and matrix Poisson varieties}, Advances in Mathematics \textbf{226}
(2011), 779-826.


\bibitem{gl-ijm} K R Goodearl and T H Lenagan, \emph{Prime ideals invariant
under winding automorphisms in quantum matrices}, International Journal of
Mathematics \textbf{13} (2002), 497-532.

\bibitem{ll} S Launois and T H Lenagan, \emph{From totally nonnegative matrices
to quantum matrices and back, via Poisson geometry}, to
appear in the Proceedings of the Belfast Workshop on Algebra, Combinatorics
and Dynamics 2009. 

\bibitem{mm} M Marcus and H Minc, A Survey of Matrix Theory and 
Inequalities, Allyn and Bacon, Boston, 1964. 


\bibitem{muir} T Muir, A Treatise on the Theory of Determinants, 
Revised and enlarged by William H. Metzler, Dover, 
New York, 1960.


\bibitem{pinkus} A Pinkus, Totally Positive Matrices, Cambridge Tracts in
Mathematics 181, Cambridge University Press, Cambridge, 2010.

\bibitem{veda} R Vein and P Dale, Determinants and Their Applications in Mathematical Physics, Applied Math. Sci. Vol. 134, Springer-Verlag, New York, 1999.

\end{thebibliography}
\end{document}